\documentclass[a4paper,11pt]{amsart}

\usepackage{amssymb}
\usepackage{fullpage}
\usepackage{enumerate}
\usepackage[OT4]{fontenc}
\usepackage{tikz} 
\usepackage{float}
\usetikzlibrary{arrows}

\theoremstyle{definition}
\newtheorem{definition}{Definition}
\theoremstyle{plain}
\newtheorem{theorem}[definition]{Theorem}
\newtheorem{proposition}[definition]{Proposition}
\newtheorem{lemma}[definition]{Lemma}

\newtheorem{observation}[definition]{Observation}

\newtheorem{property}[definition]{Property}

\def\be{\begin{equation}}
\def\ee{\end{equation}}

\def\bp{\begin{property}}
\def\ep{\end{property}}

\def\bbP{{\mathbb P}}

\def\eee12{\frac{1}{\sqrt{r+\frac{1}{12}}}}

\def\ep{\varepsilon({\cal O}_{\bbP^2}(1), P_1,...P_r)}

\begin{document}

\title{\bf A WAY FROM THE ISOPERIMERTRIC INEQUALITY IN THE PLANE TO A HILBERT SPACE}

\author{Edward Tutaj}



\begin{abstract}

 {\small We will formulate and prove a generalization of the
isoperimetric inequality in the plane. Using this inequality we
will construct an unitary space - and in consequence - an
isomorphic copy of a separable infinite dimensional Hilbert space,
which will appear  also an RHKS.}

\end{abstract}

\maketitle

\section{Introduction}

\vspace{5mm}

Let $U\subset \mathbb R^{2}$ be a compact set having well-defined
perimeter - say $o(U)$ - and a plane Lebesgue measure - say
$m(U)$. The classical isoperimetric inequality asserts, that the
following inequality holds:

\begin{equation}\label{nierklas}
o^2(U)\geq 4\pi\cdot m(U),
\end{equation}
together with the additional remark, that equality holds if and
only if $U$ is a disk.

The isoperimetric inequality  is commonly known and as one of
oldest inequalities in mathematics, it has many proofs. An
excellent review is in \cite{Tre} and some recent ideas are to be
found in \cite{Kla}. The inequality (\ref{nierklas}) may be
generalized to higher dimensions, but in this paper we will
consider only the plane case. The possibilities to generalize the
inequality (\ref{nierklas}) for some larger class of subsets of
$\mathbb R^2$ are practically exhausted, since the claiming {\it
to have a well-defined perimeter} is very restrictive. However the
convex and compact sets have a perimeter and are measurable and
the class of such sets is sufficiently large from the point of
view of applications. Moreover this class is closed with respect
to standard algebraic operation ({\it Minkowski addition } and
scalar multiplication).  This makes possible to construct in a
unique way a vector space - which will be denoded by
$X_{\mathcal{S}}$ - containing all - in our case - centrally
symmetric, convex and compact sets. We will extend the definitions
of  perimeter $o$ and the measure $m$ onto $X_{\mathcal{S}}$ and
we will prove, that for this extended operation, and for each
vector $x\in X_{\mathcal{S}}$ the {\it generalized isoperimetric
inequality} holds:

\begin{equation}\label{nieruog}
o^2(x)\geq 4\pi\cdot m(x) 
\end{equation}

Using the inequality (\ref{nieruog}) we will define an inner
product in the considered space vector space $X_{\mathcal{S}}$,
whose completion gives a model of a separable Hilbert space.

\vspace{11mm}

\section{Part I}
\subsection{  A cone of norms in $\mathbb R^2$.}

\vspace{7mm}

Let $\mathcal{S}$ denote the family of all subsets  of $\mathbb
R^2$, which are non-empty, compact  and centrally symmetric. In
$\mathcal{S}$ we consider the so-called Minkowski addition and the
multiplication by positive scalars. We recall below the
definitions by the formulas (\ref{dodMin}) i (\ref{mnozMin}).

 Assume that $V\in {\mathcal{S}}\ni W$, $\lambda \geq 0$. We set:
\begin{equation}\label{dodMin}
    V+W = \left\{v+w: v\in V, w\in W\right\},
\end{equation}
\begin{equation}\label{mnozMin}
\lambda\cdot V= \left\{ \lambda\cdot v: v\in V\right\},
\end{equation}

It is not hard to check, that  ($\mathcal{S},+, \cdot$), i.e. the
set $\mathcal{S}$ equipped with the ope\-rations defined above, is
a vector cone.

Clearly, each set $V\in \mathcal{S}$ contains the origin.
 If such a set $V$ has non-empty interior, then there is
 in $\mathbb R^2$ a norm $||\cdot||_{V}$,
 such that  $V$ is a unit ball for this norm.
 However we admit in $\mathcal{S}$ also the sets with empty interiors,
  i.e. the sets of the form:

\begin{equation}\label{definDiangl}
I({\bf{v}},d)= [-d,d]\cdot {\bf{v}},
\end{equation}
where $\bf{v}$ is a unit vector in $\mathbb{R}^2$, and  $d\geq 0$
is a scalar. The set of the type $I({\bf{v}},d)$ will be called
{\it a
 diangle}.

It is known, that in the cone ($\mathcal{S},+, \cdot$) holds the
so-called {\it cancellation law} (\cite{Rad}), i.e. the following
property:

\begin{equation}\label{prawoskracania}
 \left(V\in \mathcal{S}, U\in \mathcal{S}, W\in \mathcal{S} \right)\Longrightarrow
\left(V+W=U+W \Longrightarrow V=W\right).
\end{equation}

\vspace{3mm}

A consequence of the properties formulated above is the following
theorem called {\it the Radstr$\ddot{o}$m embedding theorem}:

\vspace{3mm}

\begin{proposition}\label{nakrywanie}

  There is a unique - up to an isomorphism  - vector space $X_{\mathcal{S}}$ "covering the cone"
  $\mathcal{S}$. This space is a quotient of the
product ${\mathcal{S}}\times {\mathcal{S}}$ by the equivalence
relation $\diamond$, where:

\begin{equation}\label{relacja karo}
 (U,V)\diamond (P,Q) \Longleftrightarrow U+Q=V+P.
\end{equation}

In  $X_{\mathcal{S}}$ one defines the operations by obvious
formulas:

\begin{equation}
[U,V]+[P,Q]=[U+P,V+Q],
\end{equation}

\begin{equation}
\lambda\geq 0\Longrightarrow \lambda[U,V]=[\lambda U,\lambda V]
\end{equation}

\begin{equation}
(-1)[U,V]=[V,U].
\end{equation}

One can check, that the operations in  $X_{\mathcal{S}}$   are
well defined (i.e. do not depend on the choice of the
representatives) and that the space $X_{\mathcal{S}}$ defined in
such a way is unique up to an isomorphism.

The map
\begin{equation}
i: \mathcal{S}\ni U\longrightarrow [(U,\left\{\theta\right\})]\in
X_{\mathcal{S}},
\end{equation}

where $\theta$ is the origin, is called the canonical embedding
and instead of $[(U,\left\{\theta\right\})]$, we will write simply
$U$ and we will also write $\mathcal{S}$ instead of
$i(\mathcal{S})$.

\end{proposition}

\vspace{3mm}

The construction described in Proposition \ref{nakrywanie} can be
realized not only in $\mathbb R^2$, but also in the case of
general Banach spaces, or even locally convex spaces and without
the claiming of central symmetry. The details are to be found in
many papers, for example in \cite{Hor}.

\vspace{3mm}

We will use in the sequel the diangles defined above. Let us
denote by $\mathcal{D}$ the set of all diangles. Clearly this set
is closed with respect to the multiplication by positive reals but
it is not a subcone of the cone $\mathcal{S}$. However we can
consider the subcone $\mathcal{W}$ generated by all diangles
understood as the smaller cone containing $\mathcal{D}$. It is
easy to check, that $W\in \mathcal{W}$ if and only if $W$ can be
written in the form:

\begin{equation}\label{stare1.7}
W= \sum_{k=1}^{n}I_k , 
\end{equation}
where $I_k\in \mathcal{D}$. In other words $\mathcal{W}$ is the
set of all centrally symmetric polygons. One can easily check,
that for each $W\in \mathcal{W}$ the representation
(\ref{stare1.7}) is unique. In other words the set of diangles is
linearly independent for finite sums.

\vspace{5mm}

\subsection{ Width functionals}\label{width}

\vspace{3mm}

We will use also the so-called {\it width functionals } on
$\mathcal{S}$. Suppose that $f:\mathbb R^{2}\longrightarrow
\mathbb R$ is a linear functional on $\mathbb R^2 $. For $V\in
\mathcal{S}$ we set:

\begin{equation}\label{deffkre}
\overline{f}(V)=\sup\left\{f(v):v\in V\right\} 
\end{equation}
It is proved in \cite{Rad}, that the map
\begin{equation}\label{addytywf}
\overline{f}:{\mathcal{S}}\ni V\longrightarrow \overline{f}(V)\in [0,\infty)
\end{equation}
is a homomorphism of cones. This means, that
$\overline{f}(U+V)=\overline{f}(U)+\overline{f}(V)$ and
$\overline{f}(\alpha \cdot V)=\alpha\cdot \overline{f}(V).$ The
linearity of the functional $\overline{f}$ will be used in many
variants in connection with the important notion of {\it the width
of a set } with respect to a given direction.

\vspace{3mm}

 Let $V\in \mathcal{S}$ be a set and let $k$ be a
straight line  containing the origin.

\begin{definition}\label{szerokosc}
The width of  $V$ with respect to  $k$ - denoted by
$\overline{V}(k)$ - is the lower bound of all numbers
$\varrho(k_1,k_2)$ where $k_1$ and $k_2$ are the straight lines
parallel to $k$, $V$ lies between $k_1$ and $k_2$ and
$\varrho(k_1,k_2)$ is the distance of lines $k_1$ and $k_2$.

\end{definition}

For each  line $k$ there exists a unique (up to the sign) linear
functional $f_k$, such that $||f_k|| = 1$ and $k$ is the kernel of
$f_k$. Clearly we have

$$\overline{V}(k)=
2\cdot \overline{f_k}(V)$$and in consequence we have the formula:
\begin{equation}\label{wlasfzkre}
\overline{U+V}(k)=\overline{U}(k)+\overline{V}(k).
\end{equation}

In particular, if  $I\in \mathcal{D}$ is a non-trivial diangle,
then $I$ determines the unique straight line $k$, such that
$I\subset k$, and hence we may speak of the width of the set $V$
with respect to the diangle  $I$, which will be denoted by
$\overline{V}(I)$.

\vspace{5mm}

\subsection{  Formulation of the main problem.}

\vspace{5mm}

Let us  consider now some natural function on the cone
${\mathcal{S}}$. Namely, let $m: {{\mathcal{B}}_2}\longrightarrow
\mathbb R$ be the two-dimensional Lebesgue measure  in the plane
defined, for simplicity, on the $\sigma$-algebra
${{\mathcal{B}}_2}$ of Borel sets in
  ${\mathbb R}^2$. Since ${\mathcal{S}}\subset {{\mathcal{B}}_2}$, then we
  can consider the restriction $m|_{\mathcal{S}}$, i.e. the map

 $$m|_{\mathcal{S}}:{\mathcal{S}}\longrightarrow [0,\infty).$$

Since in  ${ \mathcal{S}}$ we have the structure of a vector cone
described above, then it is natural to expect, that
$m|_{\mathcal{S}}$ has some algebraic properties  with respect to
the algebraic operations in ${\mathcal{S}}$.  In particular we
have: $m(\lambda \cdot U)= \lambda^2\cdot m(U)$. Much more
informations gives the very well known  {\it Steiner formula}.
Namely for $U\in {\mathcal{S}}$, for $B=B(0,1)$ - (where $B$ is
the unit disc) and for $\lambda \in [0,\infty)$ the Steiner
formula says that:

\begin{equation}
 m(U+\lambda B)= m(U) + \lambda \cdot o(U) + {\lambda}^2 \cdot m(B),
\end{equation}
where $o(U)$ is the perimeter of the convex set $U$.
 Since, as we have observed above, the cone ${\mathcal{S}}$ is a subcone
 of the covering space  $X_{\mathcal{S}}$, it is natural to expect that
 $m|_{\mathcal{S}}$ is a restriction of some  "good" function
$m^{*}:X_{\mathcal{S}}\longrightarrow \mathbb R $. This is true.
Namely we will prove that:

\vspace{5mm}

\begin{theorem}\label{rozszerzeniemiary}

  With the notations as above, there exists a unique
polynomial of the degree two,
$m^{*}:X_{\mathcal{S}}\longrightarrow \mathbb R$, such that,
$m^{*}|_{\mathcal{S}} = m|_{\mathcal{S}}$.
\end{theorem}
\vspace{5mm}

The proof of this theorem is perhaps not too difficult, but it is
rather long, since we must construct some unknown object, namely
the polynomial $m^{*}$. The crucial point is the proof of the
fact, that the necessary formula for $m^{*}$ does not depend on
the choice of the representatives. We will do the proof in a
number of steps.

\subsection{Construction of the extension of $m$ onto $X_{\mathcal{S}}$}

\vspace{5mm}

We shall start by defining two functions:
$\Psi:{\mathcal{S}}\times {\mathcal{S}}\longrightarrow \mathbb R$
and ${\mathcal{M}}:{\mathcal{S}}\times
{\mathcal{S}}\longrightarrow \mathbb R$ by the following formulas.
Let $U,V,A,B\in \mathcal{S}$. Define

\begin{equation}\label{definicjaPsi}
\Psi(U,V)= 2m(U)+2m(V)-m(U+V) 
\end{equation}

 and
\begin{equation}\label{definicjaM}
{\mathcal{M}}(A,B) = \frac{m(A+B)-m(A)-m(B)}{2} 
\end{equation}

\vspace{5mm}

\begin{proposition}\label{propozycjarownowaznosc}

 The following conditions are equivalent:

 a). For all $U,V,W \in \mathcal{S}$ we have:
$\Psi(U+W,V+W)=\Psi(U,V);$

b). For all $A_1,A_2,B\in \mathcal{S}$ we have:
${\mathcal{M}}(A_1+A_2,B)={\mathcal{M}}(A_1,B)+{\mathcal{M}}(A_2,B);$

c). For all $U,V,P,Q \in \mathcal{S}$ we have: $(U,V)\diamond
(P,Q)\Longrightarrow \Psi(U,V)=\Psi(P,Q).$

\end{proposition}

\vspace{3mm}

\begin{proof}

 Since $(U+W,V+W)\diamond (U,V)$ then $c)\Longrightarrow a).$
To prove the implication $a)\Longrightarrow b)$ let  us fix  sets
$A_1,A_2,B\in \mathcal{S}$.

Let us set: $L=2\cdot\mathcal{M}(A_1+A_2,B)$ and
$P=2\cdot\left(\mathcal{M}(A_1,B)+ \mathcal{M}(A_2,B)\right).$ It
follows from (\ref{definicjaM}) 
that:
$$L=2\cdot{\mathcal{M}}(A_1+A_2,B)=m(A_1+A_2+B)-m(A_1+A_2)-m(B) $$
and
$$P=2({\mathcal{M}}(A_1,B)+ {\mathcal{M}}(A_2,B))=m(A_1+B)+m(A_2+B)-(m(A_1)+m(A_2))-2m(B).$$

Hence

$$ L=P \Longleftrightarrow  2L=2P  \Longleftrightarrow 2m(A_1+A_2+B) - 2m(A_1+A_2) - 2m(B) =$$
$$= 2m(A_1+B)+2m(A_2+B) -(2m(A_1)+2m(A_2))- 4m(B)\Longleftrightarrow$$
$$\Longleftrightarrow 2m(A_1+A_2+B) + [2m(A_1) + 2m(A_2) -
m(A_1+A_2)]- m(A_1+A_2)=$$
$$= [2m(A_1+B) + 2m(A_2+B) - m(A_1+A_2+2B)] + m(A_1+A_2+2B)-
2m(B).$$

Now we use  Definition \ref{definicjaPsi} and we obtain the
equivalence:

$$L=P\Longleftrightarrow 2m(A_1+A_2+B)+ \Psi(A_1,A_2) - m(A_1+A_2)=$$

$$=\Psi(A_1+B,A_2+B)+ m(A_1+A_2+2B) - 2m(B).$$

Now we use the assumed condition a) and we have the equivalence:

\begin{equation}\label{stare4.4}
L=P \Longleftrightarrow 2m(A_1+A_2 +B)+2m(B)-m(A_1+A_2+2B)=m(A_1+A_2).
\end{equation}
But obviously we have:

$$2m(A_1+A_2 +B)+2m(B)-m(A_1+A_2+2B) = $$ $$= 2m(A_1+A_2+B) + 2m(\left\{\theta\right\} + B) - m(A_1+A_2+ \left\{\theta\right\} +2B) =$$
$$= \Psi(A_1+A_2+B,\left\{\theta\right\} +B) = \Psi(A_1+A_2,\left\{\theta\right\}).$$

The last equality is once more the consequence of a).

\vspace{3mm}

Using Definition (\ref{definicjaPsi}) we have
$$\Psi(A_1+A_2,\left\{\theta\right\})= 2m(A_1+A_2)+
2m(\left\{\theta\right\})- m(A_1+A_2+\left\{\theta\right\})=
m(A_1+A_2).$$ This is exactly the right hand side of  equality
(\ref{stare4.4}), which ends the proof of the implication
$a)\Longrightarrow b)$

\vspace{3mm}

Before proving $b)\Longrightarrow c)$ we shall observe some
properties of the function  $\mathcal{M}$. The condition $b)$
means, that $\mathcal{M}$ is additive with respect to the first
variable. But $\mathcal{M}$ is obviously symmetric, then
$\mathcal{M}$ is additive with respect to both  variables
separately.

Now we shall prove, that $\mathcal{M}$ is positively homogenous
(clearly with respect to both variables separately). Indeed, let
us fix two sets  $A$ and  $B$ from $\mathcal{S}$. We shall  prove
that for each $\lambda \geq 0$ there is:

\begin{equation}\label{jednorodnosc}
{\mathcal{M}}(\lambda A,B)= \lambda \cdot{\mathcal{M}}(A,B).
\end{equation}

To prove (\ref{jednorodnosc}) let us consider the function  $u(t)=
{\mathcal{M}}(tA,B)$ defined on the interval $[0,\infty)$. It is
easy to check, that the condition $ b)$ implies that $u$ satisfies
the so-called {\it Cauchy functional equation}, i.e.

$$u(t+s)= u(t)+u(s).$$

Indeed, we have:

$$u(t+s)= {\mathcal{M}}((t+s)A,B)= {\mathcal{M}}(tA+sA,B)= {\mathcal{M}}(tA,B)+{\mathcal{M}}(sA,B)=u(t)+u(s).$$

Moreover the function $u$ is locally bounded since if $\alpha =
diam(A)$, $\beta= diam(B)$, then the set $tA+B$ is contained in
the ball with the radius $t\alpha +\beta$. It follows from the
known theorem on Cauchy equation, that $u(t)=tu(1)$. This means
that $\mathcal{M}$ is positively homogenous.

\vspace{5mm}

Now we shall prove the implication  $b)\Longrightarrow c)$. Let us
fix two pairs $(U,V)$ and $(X,Y)$ such that $(U,V)\diamond(X,Y)$.
We will check that  $\Psi(U,V)=\Psi(X,Y)$.

>From the definition of the relation $\diamond$ we have: $X+V=Y+U$.
Adding $Y$ we obtain $X+Y+V = 2Y+U$. Now we add $V$ to both sides
of the last equality and we obtain: $X+Y+2V=2Y+U+V$. Thus
$m(X+Y+2V)=m(U+V+2Y)$. Now we use the definition of the function
$\mathcal{M}$ and we obtain:

\begin{equation}
m(X+Y)+2{\mathcal{M}}(X+Y,2V)+m(2V)=
m(U+V)+2{\mathcal{M}}(U+V,2Y)+m(2Y)
\end{equation}
 and  by the homogeneity:

\begin{equation}\label{stare4.8}
m(X+Y)+4{\mathcal{M}}(X+Y,V)+4m(V)=m(U+V)+4{\mathcal{M}}(U+V,Y)+4m(Y).
\end{equation}

After an analogous sequence of transformations (replacing $Y$ by
$X$ and $V$ by $U$) we obtain the equality:

\begin{equation}\label{stare4.9}
m(X+Y)+4{\mathcal{M}}(X+Y,U)+4m(U)=
m(U+V)+4{\mathcal{M}}(U+V,X)+4m(X).
\end{equation}

Adding (\ref{stare4.8}) and  (\ref{stare4.9}) and using the
bilinearity of $\mathcal{M}$ we obtain the equality:

$$2\Psi(U,V) - 2\Psi(X,Y)= 2[2m(U)+2m(V)-m(U+V)]- 2[2m(X) + 2m(Y)
-m(X+Y)]=$$
$$=4{\mathcal{M}}(U+V,X+Y)- 4{\mathcal{M}}(X+Y,U+V)=0.$$

 The last equality  implies the equality $\Psi(X,Y) = \Psi(U,V)$ and this ends
 the proof of the Proposition \ref{propozycjarownowaznosc}.
\vspace{2mm}
\end{proof}

To complete  Proposition \ref{propozycjarownowaznosc} we shall
prove that the condition a) from this Proposition is true. More
exactly we shall prove the following

\vspace{4mm}

\begin{proposition}\label{stare2}

  Let $X\in {\mathcal{S}}\ni Y, U\in {\mathcal{S}}.$ Then
\begin{equation}\label{rownaniez4}
2m(X+U)+2m(Y+U)-m(X+Y+2U)=
2m(X)+2m(Y)-m(X+Y).
\end{equation}
(i.e. $\Psi(X+U,Y+U)=\Psi(X,Y)$).
\end{proposition}

 \vspace{5mm}

\begin{proof}

 First we will check, that the equality (\ref{rownaniez4})  is true for the sets
$U$ having the form  $U=\sum_{1}^{k}I_j$, i.e. for the polygons
from the cone  $\mathcal{W}$. We apply the induction with respect
to $k$, i.e. the number of diangles representing $U$.

\vspace{3mm}

 Proof for  $k=1$. Let  $X\in {\mathcal{S}}\ni Y$ be arbitrary elements from the cone ${\mathcal{S}}$
 and let $I=I({\bf{v}},d)$ be a diangle. We have

$$L= \Psi(X+I,Y+I) = 2m(X+I) + 2m(Y+I)-m(X+Y+2I) = 2m(X) + 2m(I) + 2\cdot 2\cdot \overline{X}(I)\cdot d + $$
$$+ 2m(Y) + 2m(I) + 2\cdot 2\cdot \overline{Y}(I)\cdot d - m(X+Y) - m(2I) - 2\cdot 2\cdot \overline{X+Y}(I)\cdot d.$$

Since $m(I)=0$ and $ \overline{X+Y}(I)=  \overline{X}(I)+
\overline{Y}(I)$ (observe, that the direction of the diangle $2I$
is the same as the direction of the diangle $I$), then $L=
2m(X)+2m(Y)-m(X+Y)= \Psi(X,Y)$, which ends the proof for $k=1$.
Let us notice, that we used the {\it Cavaleri principle}. More
exactly the equality $m(X+I)= m(X)+2\cdot\overline{X}\cdot d$ and
similar equalities results from the Cavalieri principle.

\vspace{3mm}

Induction step. Assume, that our theorem is true for all sets $X$,
$Y$ and for all set $U'$ such that $U'=\sum_{1}^{k-1} d_j\cdot
I_j$. Fix now some sets  $X$ and $Y$ and a set  $U=\sum_{1}^{k}
d_j\cdot I_j$. Let us denote $X'= X+\sum_{1}^{k-1} d_j\cdot I_j$,
$Y'=Y+\sum_{1}^{k-1} d_j\cdot I_j$ and $I({\bf{v}}_k,d_k)=
I({\bf{v}},d)=I$. Then

$$2m(X+U)+2m(Y+U)-m(X+Y+2U)= 2m(X'+I)+2m(Y'+I)-m(X'+Y'+2I)=$$

(since our theorem is true for  $k=1$)

$$=2m(X')+2m(Y')- m(X'+Y') = 2m(X+U')+2m(Y+U')-m(X+Y+2U')=$$

(inductive assumption)

$$=2m(X)+2m(Y)-m(X+Y).$$

 Hence we have proved that equation (\ref{rownaniez4}) is true for all polygons $U\in {\mathcal{W}}$.
To end the proof of the Proposition \ref{stare2} for fixed sets
$X,Y,U$ we construct a sequence of polygons  $U_k$ convergent to
the set $U$ in the sense of Hausdorff distance. Since the measure
$m$ (in $\mathcal{S}$) is continuous with respect to the Hausorff
distance, then it is sufficient to pass to the limit in the
sequence of equalities:

$$2m(X+U_k)+2m(Y+U_k)-m(X+Y+2U_k)= 2m(X)+2m(Y)-m(X+Y).$$

\vspace{5mm}

\end{proof}

\subsection{ A bilinear form.}

\vspace{5mm}

Let us consider a function:$$\tilde{M}:(\mathcal{{S}}\times
{\mathcal{S}})\times({\mathcal{S}}\times {\mathcal{S}})
\longrightarrow \mathbb R
$$ defined by the formula:
\begin{equation}\label{deftildM}
\widetilde{M}((U,V),(P,Q))=
\frac{(m(U+P)+m(V+Q)-m(U+Q)-m(V+P))}{2}
\end{equation}
We shall prove that:

\vspace{3mm}

\begin{proposition}\label{stare3}
  If  $(U,V)\diamond (U_1,V_1)$ and $(P,Q)\diamond (P_1,Q_1)$
then

\begin{equation}\label{wpropoz3}
\widetilde{M}((U,V),(P,Q))=
\widetilde{M}((U_1,V_1),(P_1,Q_1))
\end{equation}
\end{proposition}

\vspace{5mm}

{\it Proof.} First we shall establish the relations between the
function $\mathcal{M}$ defined above by (\ref{definicjaM}) and the
function $\widetilde{M}$ defined by (\ref{wpropoz3}).

We check that:

$$\widetilde{M}(A,\left\{\theta\right\}),(B,\left\{\theta\right\})= {\mathcal{M}}(A,B),$$

and

$$\widetilde{M}(\left\{\theta\right\},A),(\left\{\theta\right\},B)= {\mathcal{M}}(A,B).$$

Similarly

$$\widetilde{M}(A,\left\{\theta\right\}),(\left\{\theta\right\},B)= -{\mathcal{M}}(A,B),$$

and

$$\widetilde{M}(\left\{\theta\right\},A),(B,\left\{\theta\right\})= -{\mathcal{M}}(A,B).$$

\vspace{3mm}

It is easy to check that :
\begin{equation}\label{stare5.3}
 \widetilde{M}((U,V),(P,Q))= {\mathcal{M}}(U,P)+{\mathcal{M}}(V,Q)-{\mathcal{M}}(U,Q)-{\mathcal{M}}(V,P)
\end{equation}
\vspace{3mm}

Now we observe that:

\vspace{3mm}

\begin{proposition}\label{stare5.4}
 Suppose that  $(P,Q)\diamond (P_1,Q_1)$. Then:
\begin{equation}\label{stare5.5}
{\mathcal{M}}(U,P)- {\mathcal{M}}(U,Q)= {\mathcal{M}}(U,P_1)-
{\mathcal{M}}(U,Q_1)
\end{equation}
\vspace{3mm}
\end{proposition}
\begin{proof}

Indeed (\ref{stare5.5}) is equivalent to the equality:
$${\mathcal{M}}(U,P)+ {\mathcal{M}}(U,Q_1)= {\mathcal{M}}(U,P_1)+
{\mathcal{M}}(U,Q).$$ But ${\mathcal{M}}$ is - as we have proved -
linear with respect to each variable separately, hence the last
equality is equivalent to
$${\mathcal{M}}(U,P+Q_1)= {\mathcal{M}}(U,P_1+Q).$$ The last equality is true because
$(P,Q)\diamond (P_1,Q_1)$.

Analogously, we have: $${\mathcal{M}}(V,P)- {\mathcal{M}}(V,Q)=
{\mathcal{M}}(V,P_1)- {\mathcal{M}}(V,Q_1).$$ In consequence we
obtain the equality:
$$\widetilde{M}((U,V),(P,Q))=\widetilde{M}((U,V),(P_1,Q_1))$$
and by symmetry

$$\widetilde{M}((U,V),(P_1,Q_1))= \widetilde{M}((U_1,V_1),(P_1,Q_1)).$$

\end{proof}
\vspace{3mm}

Hence  Proposition \ref{stare3} is proved. It means that the
function $\widetilde{M}$ is well defined as a function  on the
vector space $X_{\mathcal{S}}$.

It follows from  formulas (\ref{stare5.3}) that $\widetilde{M}$ is
additive with respect to each variable separately. Moreover, as we
have observed earlier, the function ${\mathcal{M}}$ is homogenous
for nonnegative scalars. Then, using (\ref{stare5.3}) we easily
check, that $\widetilde{M}$ is homogenous also for negative reals.
In consequence we may state, that $\widetilde{M}$ is a bilinear
form on $X_{\mathcal{S}}$.

In the last step we check that: $$\widetilde{M}([U,V],[U,V]) =
\Psi
([U,V]) := m^{*}([U,V])$$
which means, that the Lebesgue measure in the plane - more exactly
the function $m|_{\mathcal{S}}$ - can be extended to a polynomial
on $X_{\mathcal{S}}$.

For the uniqueness it is sufficient to observe, that when  $w$ is
a homogenous polynomial with degree $2$ then:
$$w(x+y)+w(x-y)= 2w(x)+2w(y).$$

\vspace{5mm}

\section{Part II}

\vspace{3mm}
 {\bf A generalization of the isoperimetric inequality.}

\vspace{3mm}

Now, when we have the polynomial $m^{*}$ and the bilinear form
$\widetilde{M}$, we may consider the problem of generalization of
different properties of the measure $m$ (classical) onto the
measure $m^{*}$ (generalized). In the sequel we will use for
$m^{*}$ the name {\it measure} and we will write $m$ instead of
$m^{*}$.

In this section we shall present a generalization of the
isoperimetric inequality. The classical isoperimetric inequality,
as we have written above in Introduction (\ref{nierklas}, has the
form :
\begin{equation}\label{nierkla2}
o^{2}(U)\geq 4\pi\cdot m(U),
\end{equation}where $U\in {\mathcal{S}}$. Since the
right hand side of (\ref{nierkla2}) has now a sense for $[U,V]\in
X_{\mathcal{S}}$, then if we want  generalize (\ref{nierkla2}) we
must to generalize the perimeter from $\mathcal{S}$ onto
$X_{\mathcal{S}}$.

This is not hard to do. Suppose that, as above, $U\in
{\mathcal{S}}\ni V$. Since $U$ and $V$ are, in particular convex
and closed, then they have the perimeter. Let us denote the
perimeter of the set $U\in
  {\mathcal{S}}$ by $o(U)$. It is known, that the correspondence
  $$o:X_{\mathcal{S}}\ni U \longrightarrow o(U)\in \mathbb R$$ is linear on the cone $\mathcal{S}$.
This is a consequence of the Steiner formula. Namely it is
geometrically evident, that:

\begin{equation}\label{defobwodu}
o(U) = \lim_{t\rightarrow 0}\frac{m(U+tB)- m(U) - m(B)}{t},
\end{equation}
where $B$ is a unit disk. Since $m$ is a homogenous polynomial of
the second degree, generated by the the form  $\widetilde{M}$
(bilinear and symmetric) then, it is easy to check, that:

\begin{equation}\label{defobwodu2}
 o(U)= 2\cdot \widetilde{M}([U,\left\{\theta\right\}];[B,\left\{\theta\right\}])
\end{equation}
This leads to the following:

\begin{proposition}
 The function
\begin{equation}\label{defobwodu3}
 o:X_{\mathcal{S}}\ni [(U,V)]\longrightarrow o(U)-o(V)\in \mathbb{R}
\end{equation}
  is well defined (i.e. does not depend on the representative of
  the equivalence class with respect to
$\diamond$) and is a linear functional on $X_{\mathcal{S}}$.
\end{proposition}

\begin{proof}
We set for $x=[U,V]\in X_{\mathcal{S}}$

\begin{equation}\label{defobwodu4}
o(x)=o([U,V])=2\cdot\widetilde{M}\left([U,V];[B,\left\{\theta\right\}]\right)
\end{equation}

Then we have

$$o([U,V]=2\cdot\widetilde{M}\left([U,\left\{\theta\right\}]+[\left\{\theta\right\},V];[B,\left\{\theta\right\}]\right)=$$
$$=2\cdot\widetilde{M}\left([U,\left\{\theta\right\}];[B,\left\{\theta\right\}]\right)+
2\cdot\widetilde{M}\left([\left\{\theta\right\},V];[B,\left\{\theta\right\}]\right)$$
$$=2\cdot\widetilde{M}\left([U,\left\{\theta\right\}];[B,\left\{\theta\right\}]\right)
-2\cdot\widetilde{M}\left([V,\left\{\theta\right\}];[B,\left\{\theta\right\}]\right)=o(U)-o(V).$$

This means, that the definition of the perimeter $x\longrightarrow
o(x)$ does not depend on the choice of representative and setting
$[[B,\left\{\theta\right\}]=B$ we may write:
$o(x)=2\cdot\widetilde{M}(x,B)$. Thus
$o:X_{\mathcal{S}}\longrightarrow \mathbb{R}$ is a linear
functional on $X_{\mathcal{S}}$.

\end{proof}

  \vspace{3mm}

Now we are ready to formulate  the following generalization of the
isoperimetric inequality:

\vspace {5mm}

\begin{theorem}\label{twierdzenieonieruog}

  For each vector  $x=[(U,V)]\in X_{\mathcal{S}}$ the following
inequality holds:

\begin{equation}\label{nieruog2}
o^2([(U,V)])\geq  4 \pi\cdot m([(U,V)])
\end{equation}
(let us remember that here and in the sequel we  write $m$ instead
of $m^{*}$).
\end{theorem}
Let us observe that (\ref{nieruog2}) agree with (\ref{nieruog}).
Let us observe also, that this is in fact a generalization of the
classical isoperimetric inequality, since when $V$ is trivial ($V
= \left\{\theta \right\}$) then (\ref{nieruog2}) gives
(\ref{nierkla2}).

The proof of Theorem \ref{twierdzenieonieruog} is not quite
trivial since now the right hand side of (\ref{nieruog2}) can be
negative. In other words,  inequality (\ref{nieruog2}) does not
remain true, when one replaces $m([(U,V)]$ by its absolute value
(i.e. by $|m([(U,V)]|)$. We will need a number of lemmas.

 \vspace{3mm}

\begin{lemma}\label{stare6.6}

 Suppose that the sets $U\in {\mathcal{S}}\ni V$
 and suppose that $V\in \mathcal{W}$ is represented as follows:
$$V=\sum_{j=1}^{k}I_j,$$
 where $I_j= I({\bf{v_j}},d_j)$ and
 $d_j=d(I_j)$. Then we have:

\begin{equation}\label{rownaniez6.6}
 m(U+V) = m(U) + m(V) + 2\cdot\sum_{j=1}^{k}\overline{U}(I_j)\cdot d(I_j).
\end{equation}
where $\overline{U}$ is defined by (\ref{wlasfzkre}).

\end{lemma}

\begin{proof}

 We shall prove this formula inductively with respect to the
number $k$ of diangles in the representation of $V$. For $k=1$,
using the Cavalieri principle, we have:

$$m(U+I)= m(U) + 2\cdot \overline{U}(I) \cdot dI. $$

But in our case we have $m(V)=m(I)=0$, hence

\begin{equation}\label{pierwszykrok}
m(U+I)= m(U) +m(I) + 2\cdot \overline{U}(I)\cdot dI 
\end{equation}
and this means that the formula (\ref{rownaniez6.6}) is true for
$k=1$.
\vspace{3mm}

Suppose now that   (\ref{rownaniez6.6}) is true for each set $U\in
\mathcal{S}$ and for each   $2k$-angle $V=\sum_{j=1}^{k}I_j$. Let
us fix an arbitrary $U\in \mathcal{S}$ and $2(k+1)$-angle
$V=\sum_{j=1}^{k+1}I_j$.
 We have:

$$ m(U+V)= m(U+\sum_{j=1}^{k+1}I_j) = m((U+ \sum_{j=1}^{k}I_j)+
I_{k+1})$$

Now we apply   (\ref{stare6.6}) in the case  $k=1$ for  $U'=U+
\sum_{j=1}^{k}I_j$ and $V=I_{k+1}$ and we obtain that the above
equals:
$$m(U + \sum_{j=1}^{k}I_j) + 2 \cdot (\overline{U+
\sum_{j=1}^{k}I_j})(I_{k+1})\cdot dI_{k+1} $$

Now we apply  (\ref{stare6.6}) for $k$ (inductive assumption ) and
the additivity of the width of a set (formula (\ref{wlasfzkre}))
with respect to the addition in  $\mathcal{S}$ and we obtain that
the above equals:
$$ m(U) + m(\sum_{j=1}^{k}I_j) + 2\cdot
\sum_{j=1}^{k}\overline{U}(I_j)\cdot dI_j+ 2\cdot
\overline{U}(I_{k+1})\cdot dI_{k+1} + 2\cdot
\overline{(\sum_{j=1}^{k}I_j)}(I_{k+1})\cdot dI_{k+1}. $$

Using once more the property  (\ref{stare6.6}) in the case $k=1$
applied to $U= {\sum_{j=1}^{k}I_j}$ and $I=I_{k+1}$ we get that
the above equals:

$$m(\sum_{j=1}^{k}I_j) + 2\cdot \overline{(\sum_{j=1}^{k}I_j)}(I_{k+1})\cdot
dI_{k+1}= m(\sum_{j=1}^{k+1}I_j) = m(V).$$

Finally this equals

$$= m(U) + m(V) + 2\cdot \sum_{j=1}^{k+1}\overline{U}(I_j)dI_j .$$

This ends the proof of the Lemma \ref{stare6.6}.

\end{proof}

\vspace{5mm}

\subsection{ Polygons in singular position}

\vspace{3mm}

As we observed above each polygon  $W\in {\mathcal{W}}\subset
\mathcal{S}$ is a Minkowski sum of a number of diangles. Each
diangle  $I$ has a form $I=I({\bf{v}},d)=[-d,d]\cdot {\bf{v}}$.
The direction of the vector  ${\bf{v}}$ will be called the
direction of the diangle $I$. We will say that two diangles are
parallel, when they have the same direction.

\vspace{3mm}
\begin{definition}\label{singpos}
 Let us consider two polygons from $\mathcal{W}$ and let us denote
them $U=\sum_{i=1}^{n}J_i$ and $V=\sum_{j=1}^{k}I_j$. We will say
that $U$ and  $V$  are in {\it
 singular position}, when there exists  $i_0\leq n$ and
$j_0\leq k$, such that the diangles  $J_{i_0}$ and $I_{j_0}$ are
parallel.
\end{definition}
In other words the singular position means that at least one  side
of the polygon $U$ is parallel to a side of the polygon $V$.

\vspace{3mm}

Now we will define a function, which will be called {\it a
rotation function}.  For a given angle $\varphi\in [0,\pi)$ we
denote by  $O(\varphi)$ the rotation of the plane  $\mathbb R^2$
determined by the angle  $\varphi$.
 The image of the set   $W\in\mathcal{S}$ by  $O(\varphi)$, i.e.
 the set   $O(\varphi)(W)$ will be denoted by  $W^{\varphi}$.
 Since the rotations are linear,  they preserve the Minkowski
 sums. In particular, if   $V=\sum_{j=1}^{k}I_j$, then

\begin{equation}\label{wzornaobroty}
V^{\varphi}= \sum_{j=1}^{k}I_j^{\varphi}.
\end{equation}

\vspace{3mm}

Let us fix now two polygons  $U=\sum_{i=1}^{n}J_i$ and
$V=\sum_{j=1}^{k}I_j$. We define a function ({\it a rotation
function}): ($\varphi\in [0,\pi]$)

\begin{equation}\label{deffunrot}
 E(\varphi) = m(U+V^{\varphi})
\end{equation}

Since $V$ is centrally symmetric, then we have $V^{\varphi+\pi} =
V^{\varphi}$. It follows from this equality, that the function $E$
can be considered, if necessary, as a periodic function on whole
$\mathbb R$. It is also clear, that for fixed $U$ and $V$ the
function $$\mathbb R\ni \varphi\longrightarrow E(\varphi)\in
\mathbb R$$ is continuous.

When the polygon  $V^{\varphi}$ changes its position with
$\varphi$ and the polygon $U$ remains fixed, then, in general for
more than one $\varphi$ the polygons  $U$ and  $V^{\varphi}$ are
in singular position.

\vspace{5mm}

We shall prove the following lemma.

\vspace{5mm}
\begin{lemma}\label{pozsingul}
If ${\varphi}_0\in [0,\pi]$ is a point, in which the function $E$
attains its minimum, then the pair
 $(U,V^{\varphi_0})$ is in singular position.
\end{lemma}
\vspace{3mm}

 It follows from  (\ref{rownaniez6.6}) and  (\ref{wzornaobroty})  that
$$E(\varphi)= m(U+V^{\varphi})=m(U)+
m(V^{\varphi}) + 2\cdot
 \sum_{j=1}^{k}\overline{U}({I_j}^{\varphi})\cdot d({I_j}^{\varphi}).$$

Hence, if $E(\varphi_0)\leq E(\varphi)$  we obtain the inequality:

$$E({\varphi}_0)= m(U)+
m(V^{\varphi_0}) + 2\cdot
 \sum_{j=1}^{k}\overline{U}(I_{j}^{\varphi_0})\cdot d(I_{j}^{\varphi_0}) \leq$$ $$\leq E(\varphi)= m(U+V^{\varphi})=m(U)+
m(V^{\varphi}) + 2\cdot
 \sum_{j=1}^{k}\overline{U}(I_{j}^{\varphi})\cdot d(I_{j}^{\varphi}).$$

Since  $m(V^{\varphi_0}) = m(V^{\varphi})$, (the measure $m$ is
invariant under rotations) and since
 $dI^{\varphi}$ does not depend on $\varphi$, we conclude that the inequality
  $E(\varphi_0)\leq E(\varphi)$ is equivalent to the inequality:

\begin{equation}\label{stare7.6}
\sum_{j=1}^{k}\overline{U}(I_{j}^{\varphi_0})\cdot d(I_{j})\leq
 \sum_{j=1}^{k}\overline{U}(I_{j}^{\varphi})\cdot d(I_{j}).
\end{equation}

Let us denote
\begin{equation}\label{stare7.7}
 F(\varphi)=
\sum_{j=1}^{k}\overline{U}(I_{j}^{\varphi})\cdot
d(I_{j}).
\end{equation}

Clearly, since the difference of   $F(\varphi)$ and  $E(\varphi)$
is constant,  $F$ attains its absolute minimum at the same point
as the function $E$, i.e. at $\varphi_0$. Hence it is sufficient
to show an equivalent form of  Lemma \ref{pozsingul}. Namely:

 \begin{proposition}\label{wystardolemma}
If $F$ attains its absolute minimum at $\varphi_0$, then the pair
$(U,V^{\varphi_0})$ is in singular position.
\end{proposition}
\vspace{3mm}

To prove Proposition \ref{wystardolemma} we will need some
observations concerning the so-called {\it interval-wise concave}
functions.

\subsection{ Interval-wise concave functions}

\vspace{5mm}
 Consider a function $H:\mathbb R\longrightarrow \mathbb R$, which is
 continuous and periodic with period $\pi$. Hence, in particular
$H(0)=H(\pi)$.

\vspace{5mm}
\begin{definition}\label{defconc}
 We will say that a function  $H$ as above is {\it interval-wise
concave} when there exists a sequence $(\alpha_k)_{k=0}^{n}$ such
that:

  i). $0=\alpha_0<\alpha_1<....<\alpha_k=\pi$,

 ii). For each  $0\leq i \leq k-1$ the restriction of the function
$H$
  to the interval $[\alpha_i,\alpha_{i+1}]$ is concave.

\end{definition}

\vspace{3mm}

We will need the following simple properties of the interval-wise
concave functions:

\begin{proposition}\label{wlasconc}

a). If $H$ is interval-wise concave and $\lambda$ is a
non-negative scalar, then $\lambda\cdot H$ is interval-wise
concave.

b). The sum of two  (or a finite number) interval-wise concave
functions is an interval-wise concave function.

\end{proposition}

\vspace{3mm}

\begin{proof}

The property a) is obvious, since the product of a concave
function by a non-negative real is still concave.

In the proof of the property b) we will use the language from the
Riemman integral theory. We will call {\it the nets } the
sequences of the type $0=\alpha_0<\alpha_1<....<\alpha_k=\pi$. If
we have two such nets, say $\alpha=(\alpha_i)_{0}^{k}$ and
$\beta=(\beta_j)_{0}^{m}$, then we define a net
$\gamma=(\gamma_l)_{0}^{s}$ as the union of points of the nets
$\alpha$ and $\beta$ with natural ordering. In such a case we will
say, that $\gamma$ is {\it finer} than $\alpha$ (and clearly also
than $\beta$) or, equivalently, that $\alpha$ is a {\it subnet} of
the net {\it gamma}. It is also clear, that if a function $H$ is
concave in each interval of a net $\alpha$ and $\gamma$ is finer
than $\alpha$, then $H$ is concave in each interval of the net
$\gamma$. Hence for each interval-wise concave function $H$ there
exists a net $\alpha_H$, which is {\it maximal} with respect to
$H$ in the following sense: $H$ is concave in each subinterval of
the net $\alpha_H$ and if $H$ is concave in the subintervals of a
net $\gamma$, then $\gamma$ is finer than $\alpha_H$.

Suppose now, that we have two interval-wise concave functions $H$
and $G$. Suppose, that the function $H$ is concave in the
intervals $[\alpha_i,\alpha_{i+1}]$ of a net
$\alpha=(\alpha_i)_{0}^{k}$ and the function $G$ is concave in the
intervals $[\beta_j,\beta_{j+1}]$ of a certain net
$\beta=(\beta_j)_{0}^{m}$. Let $\gamma=(\gamma_l)_{0}^{s}$ be the
union of the nets $\alpha$ and $\beta$. Each subinterval
$[\gamma_l,\gamma_{l+1}]$ is the intersection of the subintervals
of the type $[\alpha_i,\alpha_{i+1}]$ and $[\beta_j,\beta_{j+1}]$.
Hence for each $r\leq  s$ the functions
$H|[\gamma_r,\gamma_{r+1}]$ and $G|[\gamma_r,\gamma_{r+1}]$ are
both concave. In consequence the function $H+G$ is concave in the
subintervals of the net $\gamma=(\gamma_l)_{0}^{s}$, and this ends
the proof of the property b).

\end{proof}

\vspace{5mm}

Let us observe that:

\vspace{5mm}

\begin{proposition}\label{wklediangl}

Let  $J$ and  $I$ be two diangles.
 Then the function  $\overline{J_{I^\varphi}}$ considered on the
 interval $[-\frac{\pi}{2},\frac{\pi}{2}]$ is interval-wise
 concave.
\end{proposition}
\begin{proof}

 The property to be interval-wise concave is invariant with respect
to the rotations of the set $J$ and with respect to the scalar
multiplication. Hence without loss of generality we may assume
that  $J=[-1,1]\cdot (1,0)$. For the same reason we may assume,
that  $I=[-1,1]\cdot (0,1)$. In such a case, as it is easy to
check, that
$$\overline{J_{I^\varphi}}=\overline{J_{I}}(\varphi) = \cos(\varphi).$$

  This ends the proof of  Proposition \ref{wklediangl} since the function
   cosinus is concave in the interval  $[-\frac{\pi}{2},\frac{\pi}{2}]$.

\end{proof}
Let us mention here, that for each function of the type
$\overline{J_{I^\varphi}}$ there exists a point $\alpha$ such that
$\overline{J_{I^\varphi}}$ is concave in each interval
$[-\frac{\pi}{2},\alpha]$ and $[\alpha,\frac{\pi}{2}]$ and is not
concave in any neighbourhood of the point $\alpha$ and $\alpha$ is
a point such that $J$ is parallel to $I^{\alpha}$.

\vspace{3mm}
\begin{proposition}\label{wklesloscF}

 Let $U=\sum_{k=1}^{m}J_k$ be a polygon from $\mathcal{W}$ and let
$I$ be a diangle. Then the function $\overline{U_{I^\varphi}}$ is
interval-wise concave.  In consequence the function $F$ defined by
the formula (\ref{stare7.7}) is interval-wise concave.
\end{proposition}

\vspace{3mm}
\begin{proof}
  It follows from the properties proved above that

$$\overline{U_{I}}(\varphi) = \overline{\sum_{k=1}^{m}({J_k})_{I}}(\varphi)=
\sum_{k=1}^{m}\overline{({J_k})_I}(\varphi).$$

By  Proposition \ref{wklediangl} each of functions
$\overline{({J_k})_I}$ is interval-wise concave and this ends the
proof of Proposition \ref{wklesloscF}.

\vspace{5mm}

\end{proof}

Now we are ready to prove  Proposition \ref{wystardolemma} and at
the same time  Lemma \ref{pozsingul}.

\begin{proof}
Let $\varphi_0$ be an argument, at which the function $F$ attains
its absolute minimum. Clearly, such a point exists, since the
function $F$ is continuous and periodic. Let
$\alpha=(\alpha_i)_{0}^{k}$ be a net such that $F$ is concave in
each  interval $[\alpha_i, \alpha_{i+1}]$. Clearly, we may assume,
that this net is maximal (i.e. $\alpha= \alpha_F$), which means,
that $F$ is not concave in any neighbourhood of the points
$\alpha_i$. Since a concave function cannot attain its minimum at
the interior point of the interval in which it is defined, then
there exists $0<\alpha_i<\pi$ such that $\alpha_i=\varphi_0$. But
the ends of the intervals of the net $\alpha=(\alpha_i)_{0}^{k}$
have such property, that one of diangles $I_j$ is parallel to the
straight line  joining two successive vertexes of the polygon U,
i.e. is parallel to a diangle  $J_i$. This means that $U$ and $V$
are in singular position.
\end{proof}

\vspace{5mm}
\subsection{  The proof of the generalized isoperimetric inequality}

\vspace{3mm}

Let us begin by the following observation.

\vspace{3mm}

\begin{observation}\label{stare9.1}
 Suppose that  $U\in {\mathcal{S}}\ni V$ are two polygons, where
$U=\sum_{i=1}^{n}J_i$ and $V=\sum_{j=1}^{k}I_j$. Then there exists
a pair of polygons  $U'$ and $V'$ such that:

 i). The perimeter of the pair   $(U,V)$ equals to the perimeter of the pair
   $(U',V')$;

 ii). The joint number of sides of the pair  $(U',V')$, understood
 as the sum of the number of sides of $U'$ and $V'$,
   is strictly less that the joint number of sides of the pair $(U,V)$;

 iii). The measure  $m([U,V])$ is less or equal than the measure
  $m([U',V'])$.

\end{observation}
\begin{proof}

 To prove  Observation \ref{stare9.1} we set   $U'=U$ and $V'=
  V^{\varphi}$, where the angle ,  $\varphi$ is such, that the pair j  $(U,V^{\varphi})$
  is in singular position. More exactly, $\varphi$ is such that the function
  $F$ attains absolute minimum exactly at    $\varphi$. Now we see, that
   condition i) is fulfilled since $O(\varphi)$ is an isometry.

  To prove iii) we observe that

  $$m([U,V]) = 2\cdot m(U) + 2\cdot m(V) - m(U+V)$$
 and

  $$m([U',V'])= 2\cdot m(U') + 2\cdot m(V') - m(U'+V')= 2m(U) +
  2m(V^{\varphi}) - m(U+V^{\varphi}).$$
 Hence $$m([U,V])\leq m([U',V'])\Longleftrightarrow
  m(U+V^{\varphi})\leq m(U+V).$$
But  $$m(U+V^{\varphi})= m(U) + m(V^{\varphi}) + 2\cdot
  F(\varphi)$$
  and  $$m(U+V) = m(U) + m(V) + 2\cdot F(0).$$ Thus

$$m(U+V^{\varphi})\leq m(U+V)\Longleftrightarrow F(\varphi)\leq F(0).$$

The last inequality is true because of the choice of  $\varphi$.

Moreover we know that  $U'$ and  $V'$ have a pair of parallel
sides. This means that there exists a polygon  $U''$ with
$(2n-2)$-angles and a diangle  $I$ generated by a unit vector
 such that $U'= U''+ d_1\cdot I$, and there exists a polygon   $V''$
 with $(2k-2)$-angles, such that $V'= V'' +d_2I$. Without  loss of generality we
 may assume that $d_2\leq d_1$ since in the opposite case the argument is analogous.
 If   $d=d_1-d_2$ then the pair   $(U',V')$ is equivalent to the pair
$(U''+d\cdot I, V'')$. But the joint number of sides of this last
pair is strictly less than the joint number of sides of the pair
$(U,V)$. This ends the proof of ii) and in consequence the proof
of  Observation \ref{stare9.1}.

\vspace{3mm}
\end{proof}

\begin{observation}\label{stare9.2}

 Let $U$ and $V$ be two polygons as in  Observation
\ref{stare9.1}. Then there exists a polygon $W$ such that
$o([U,V])= o(W)$ and $m([U,V])\leq m(W)$.
\end{observation}

\vspace{3mm}

\begin{proof}

 We can continue  the reduction - described in the proof of Observation \ref{stare9.1} - of the joint
number of sides of the pair $(U,V)$, which preserves the perimeter
and increases  the measure to the moment, when one of successively
constructed polygons became trivial. But in this case the last of
constructed pairs of type   $(U',V')$ will be equivalent to the
pair  $(W,\left\{0\right\})$. This ends the proof of Observation
\ref{stare9.2}.

\end{proof}
 \vspace{3mm}

\begin{observation}\label{stare9.3}

  For each pair of polygons $U,V$ from $\mathcal{W}$ the
following inequality holds:

$$ (o(U)-o(V))^2\geq 4\pi \cdot m([U,V]).$$
\end{observation}
\begin{proof}
 Let us fix a pair of polygons $(U,V)$. Using  Observation
 \ref{stare9.2} we choose a polygon   $W$ satisfying the properties formulated
in Observation \ref{stare9.2}. Then we have:

$$(o(U)-o(V))^2 = o^2(W) \geq 4\pi \cdot  m(W) \geq 4\pi\cdot
m([U,V]).$$

The inequality $o^2(W) \geq 4\pi \cdot  m(W)$ follows from the
classical isoperimetric inequality.
\end{proof}
\vspace{3mm}

Now we are able to finish the proof of the generalized
isoperimetric inequality formulated in Theorem
\ref{twierdzenieonieruog}.

\begin{proof}

 Let us fix a pair $(U,V)\in
X_{\mathcal{S}}$. It is known, that polygons are dense in
$\mathcal{S}$ with respect to the Hausdorff distance (i.e.
$\mathcal{W}$ is dense in $\mathcal{S}$.) We choose two sequences
of polygons $(U_k)_0^{\infty}$ and $(V_k)_0^{\infty}$ such that
$U_k\longrightarrow U$ and $V_k\longrightarrow V$ in the sense of
Hausdorff metric. Moreover the functions of perimeter $o$ and
measure $m$ are continuous with respect to the considered
convergence (\cite{Mosz}). Since, by Observation \ref{stare9.3},
the generalized isoperimetric inequality is true for polygons
(pairs from $\mathcal{W}$), so we have the following sequence of
inequalities:

 $$ o^{2}([U_k,V_k])\geq 4\pi m([U_k,V_k]).$$

or more exactly

 $$ ((o(U_k)-o(V_k))^{2}\geq 4\pi (2m(U_k)+2m(V_k)-m(U_k+V_k)).$$

Passing to the limit and using the independence of $o$ and $m$ on
the choice of representatives, we obtain Theorem
\ref{twierdzenieonieruog}.

\end{proof}

\vspace{3mm}

 \subsection{ The problem of equality in the generalized isoperimetric
 inequality}

 \vspace{3mm}

It is well known, that in the classical isoperimetric inequality
(\ref{nierklas})
$$o^2(U)\geq 4\pi m(U),$$ where $U\in {
\mathcal{S}}$, the equality holds if and only if $U$ is a disc.
One may say equivalently that $o^2(U) = 4\pi m(U)$ if and only if
$U = \lambda B$ where $B$ is a unit disc and $\lambda$ is a
non-negative real. We shall prove an analogous result for the
generalized isoperimetric inequality (\ref{nieruog}). Namely we
have the following:

\begin{theorem}\label{rownosc}
    The equality in the generalized
   isoperimetric inequality $o^2(x) \geq 4\pi m(x)$ (which is
   valid, as we have proved above
   for $x\in X_{\mathcal{S}}$) holds if and only if $x$ belongs to the
   one dimensional subspace generated by the unit disc $B$. In other
   words
$$o^2(x)=4\pi m(x)$$ if and only if there exists a real  (not
necessarily positive)
   $\lambda\in \mathbb R$ such that $x=\lambda B$.
\end{theorem}

We start be recalling a well known result concerning the quadratic
forms.

\begin{lemma}\label{schwarz}
 Suppose that $\varphi: X\longrightarrow \mathbb R$ is a
quadratic form such that $\varphi(x)\geq 0$ for each $x$, and let
$\phi:X\times X\longrightarrow \mathbb R$ be a bilinear, symmetric
form, such that $\phi(x,x)=\varphi(x)$. Then the following
inequality (Schwarz inequality) holds:
\begin{equation}\label{nierSchw}
|\phi(x,y)|\leq
\sqrt{\varphi(x)}\cdot\sqrt{\varphi(y)}.
\end{equation}

\end{lemma}

\vspace{3mm}

Let us consider (for $x\in X_{\mathcal{S}}$) the so called {\it
deficit term}:
\begin{equation}\label{deficit}
D(x)= o^2(x)-4\pi m(x).
\end{equation}

The function $D$ is a quadratic form on
 ${X}_{\mathcal{S}}$ and by the generalized isoperimetric inequality we have
 $D(x)\geq 0$. Let $\varepsilon(x,y)$ be a bilinear, symmetric form generating
 $D$. It is easy to check, that for $x=[U,V]$ and $y=[P,Q]$ we
 have:
\begin{equation}\label{deficitbis}
 \varepsilon([U,V];[P,Q]) = o([U,V])\cdot o([P,Q]) - 2\pi
 (m(U+P)+ m(V+Q)- m(U+Q)- m(V+P))
\end{equation}

Then by (\ref{nierSchw}) we have:

\vspace{3mm}

 \begin{proposition}\label{SchwarzDefi}
For the quadratic form $D$ and the bilinear form $\varepsilon$
defined by (\ref{deficit}) and (\ref{deficitbis}) the following
inequality holds:

\begin{equation}\label{defSchwarz}
\varepsilon(x,y)\leq \sqrt{D(x)}\cdot\sqrt{D(y)}.
\end{equation}
\end{proposition}
\vspace{5mm}

Now we shall prove the next lemma. Namely

\vspace{3mm}
\begin{lemma}\label{homot}

 Suppose that $x\in {\mathcal{S}} \ni y$ are such, that:

   a). $o(x)=o(y),$

   b). $D(x-y)=0$.

 Then  $x$ i $y$ are homothetic, i.e. there exists  $\lambda\in
 \mathbb R$ such that  $x=\lambda y$.
\end{lemma}
\vspace{5mm}
\begin{proof}
 Suppose, that $x$ and $y$ are as above. It follows from b) that:

$$0=D(x-y)= D(x)-2\varepsilon(x,y)+D(y),$$
hence $$2\varepsilon(x,y)= D(x)+D(y),$$ and since $D\geq 0$ we
have:

$$0\leq D(x)+D(y)=2\varepsilon(x,y)\leq
2\cdot \sqrt{D(x)}\cdot\sqrt{D(y)}.$$

Thus $$(\sqrt{D(x)}-\sqrt{D(y)})^2\leq 0$$ and in consequence
$D(x)=D(y).$

\vspace{3mm}

Now we use the condition a), i.e. $o(x)=o(y)$. We have
:$$0=D(x-y)= o^2(x-y)-4\pi m(x-y)$$ and since $o(x-y)=0$ then
$m(x-y)=0$. This means that
$$2m(x)+2m(y)-m(x+y)=0.$$

The equality $D(x)=D(y)$ implies $$o^2(x)-4\pi m(x)=o^2(y)-4\pi
m(y).$$ Thus, since  $o(x)=o(y)$ then $m(x)=m(y)$. Now since
$m(x)\geq 0$ i $m(y)\geq 0$ (since $x\in\mathcal{S}\ni y$) then we
have:

$$4m(x)=2m(x)+ 2m(y) = m(x+y),$$ i.e.

$$2\sqrt {m(x)}= \sqrt{m(x+y)}.$$

But $x$ and $y$ are from $\mathcal{S}$, so we are able to apply
the Brunn-Minkowski equality and we have:

$$\sqrt{m(x+y)}\geq \sqrt {m(x)} + \sqrt {m(y)} = 2\sqrt {m(x)}=\sqrt{m(x+y)}.$$

In consequence for the considered vectors $x$ and $y$ in the
Brunn-Minkowski inequality the equality holds. It is known that in
such a case $x$ and $y$ are homothetic. Here and in the sequel we will need
some information on the Brunn-Minkowski inequality, which are to be found
for example in (\cite{Gard}). 
\end{proof}

\vspace{3mm}

Now we are ready to prove  Theorem \ref{rownosc}.

\begin{proof}
Suppose, that $w=[U,V]$ is such, that $D([U,V])=0$. Let $z=w-rB$,
where $r$ is such that $o(z)=0$. In other words $r$ is such, that
$o([U,V])=2\pi r$. We shall calculate the error term of $z$,
 namely

$$D(z)=D(w)-2r\varepsilon(w,B)+r^2D(B).$$

Clearly $D(B)=0$ and by our assumption $D(w)=0$, hence

$$D(z)=-2r\varepsilon(w,B).$$

 We know the exact formula (\ref{deficitbis}) for $\varepsilon (u,v)$. Namely

$$\varepsilon([U,V];[B,\left\{{\theta}\right\}])= o([U,V])\cdot
o([B,\left\{{\theta}\right\}])-2\pi(m(U+B)+m(V)-m(U)-m(V+B))=$$

$$2\pi o([U,V])-2\pi (m(U)+o(U)+\pi +m(V)-m(U) - m(V) - o(V)
-\pi)=$$ $$=2\pi o([U,V])-2\pi o([U,V])=0.$$

But we can write $z=[U,V]-r[B,\left\{{\theta}\right\}]$, thus
$$z=[U,V+rB].$$ Hence we have $D(U-(V+rB))=0$, $o(U-(V+rB))=0$ and
both $x=U$ and $y=V+rB$ are from $\mathcal{S}$. Then it follows
from  Lemma \ref{homot}, that for some real $\lambda$ there is:
$V+rB=\lambda U$. But $o(U)-o(V) = 2\pi r$, and from the linearity
of $o$ we have $o(V)+ 2\pi r= \lambda o(U)$. In consequence $o(V)
+ o(U)- o(V) = \lambda o(U)$. This means that $\lambda = 1$, i.e.
$U=V+rB$ or $[U,V]=r[B,\left\{{\theta}\right\}]$ and this ends the
proof of  Theorem \ref{rownosc}.

\end{proof}

\vspace{5mm}

\section{Part III}

\vspace{5mm}
 {\bf A generalization of the Brunn-Minkowski inequality}

\vspace{5mm}

The classical Brun-Minkowski inequality, used in previous section,
says  in particular that for each $U\in\mathcal{S}\ni V$ the
following inequality holds:
\begin{equation}\label{BrunMink2}
\sqrt{m(U+V)}\geq \sqrt{m(U)}+\sqrt{m(v)}
\end{equation}
In this section we will formulate and prove an inequality
(\ref{BrunMink}), which may be considered as  the Brun-Minkowski
inequality for the generalized measure $m^{*}$.

\vspace{5mm}

\subsection{ Some remarks on quadratic forms  } 

\vspace{5mm}

We shall start this chapter by recalling some properties of
quadratic forms on real vector spaces.

\vspace{3mm}

  1. Let $X$ be a real vector space and let $\eta:X\longrightarrow \mathbb
R$ be a quadratic form on $X$, i.e. $\eta$ is a homogeneous
polynomial of the second degree. This means that there exists a
bilinear, symmetric  form $\widetilde{N}:X\times X\longrightarrow
\mathbb R$ such that  $\eta(x) =
\widetilde{N}(x,x)$. 

\vspace{2mm}

 2. In the notations as above, a form $\eta$ is said to be
positively (negatively) defined, when $\eta(x)=0$ implies
$x=\theta$. We will say also, that $\eta$ is {\it elliptic.}
Equivalently, "elipticity" means, that the set of values of $\eta
$ is $[0,\infty)$ or $(-\infty,0]$.

 We will also consider the
indefinite forms, i.e. forms for which $\eta:X\longrightarrow
\mathbb R$ is surjective. In such a case we will also say that the
form $\eta$ is {\it hyperbolic.} Clearly, a form $\eta$ is
hyperbolic, if and only if  there exists two vectors $x\in X\ni y$
such that $\eta(x)>0$ and $\eta(y)<0$.

\vspace{2mm}



3. In this paper we will consider the quadratic forms, which will
be hyperbolic, but of some special type i.e. satisfying some
additional property.  Before defining this property, let us
observe, that when we have a quadratic form $\eta:X\longrightarrow
\mathbb R$ and we take into account any subspace $Y\subset X$,
then the restriction   $\eta|_{Y}$ is a quadratic form on $Y$. If
$\eta$ is of elliptic type, then for each  $Y$ the form
$\eta|_{Y}$ is  elliptic. But in the case when $\eta$ is
hyperbolic, then  $\eta|_{Y}$ in general may not be
hyperbolic. 

\vspace{2mm}

4. Let $u\in X \ni v$ be two vectors, which are linearly
independent. Let $Y(u,v):=Lin(u,v)$  be a two dimensional subspace
spanned by $u$ and $v$.  We will say, that a quadratic form $\eta$
is $(u,v)-hyperbolic$, when  $\eta|_{Y(u,v)}$ is hyperbolic. Let
us consider the situation as above. Let $\eta$ be a quadratic form
on $X$. We will prove the following lemma:

\vspace{3mm}

 \begin{lemma}\label{hyperbol}
Suppose that the following conditions are fulfilled:

1.There is a vector $b\in X$, such that the form $\eta$
  is positively defined on the one dimensional subspace $\mathbb R\cdot
  b$,and

2. There is a linear functional $b^{*}$ on $X$ such that $\eta$ is
  negatively defined on the subspace $Y= ker {b}^{*}$.

Then the form $\eta$ is hyperbolic on each plane $L(u,v)$
generated by two linearly independent vectors $u$ and $v$, such
that $\eta(u)> 0$ or $\eta(v)> 0$.
\end{lemma}
\begin{proof}

 Suppose, that $u$ and $v$ are two linearly independent vectors such that
 $\eta(u)> 0$ and $\eta(v)> 0$.  Consider the straight line
 $\mathbb R \ni t \longrightarrow u+t\cdot v$.
We check, that there exists a vector $w\in L(u,v)$, such that
${b}^{*}(w)=0$. Indeed, it is sufficient to take
$$t=\frac{-{b}^{*}(u)}{{b}^{*}(v)}.$$

It follows from our  assumptions that $\eta(w)< 0$ and thus
$L(u,v)$ contains two vectors, namely $u$ and $w$, such that
$\eta(u)>0$ and $\eta(w)<0$. This is sufficient for the form
$\eta$ to be hyperbolic on $L(u,v)$.
\end{proof}

\vspace{3mm}

\begin{observation}\label{stare11cos}
 Let $u\in X\ni v$ be such, that $\eta(u)>0$ and $\eta(v)>0$, and
let $\widetilde{N}$ be a bilinear symmetric form  generating
$\eta$ (i.e. $\widetilde{N}(x,x)= \eta(x)$). Then the following
inequality holds:
$${\widetilde{N}}^{2}(u,v)\geq \eta(u)\cdot \eta(v).$$

\end{observation}

\begin{proof}
Indeed, we know from Lemma \ref{hyperbol}, that the quadratic
equation $\eta(u+t\cdot v)=0$ has a solution. This equation can be
written in the form:
$$\eta(u)+2\cdot t\cdot \widetilde{N} + t^{2}\cdot \eta(v) = 0.$$
Hence we have $$\Delta =  (2\widetilde{N}(u,v))^{2} -
4\eta(u)\eta(v) \geq 0$$ and this ends the proof.
\end{proof}

\vspace{5mm}

\subsection{ A corollary for the  generalized Lebesgue measure}

\vspace{5mm}

 Let, as above, $m$ denote the  Lebesgue measure on $X_{\mathcal{S}}$, which
 is, as we know, a quadratic form. Let $B$ denote the unit disc.
It is easy to check, that $m$ satisfies the assumptions of the
Lemma \ref{hyperbol}. Indeed $m$ is positively defined on one
dimensional subspace generated by $b=B$ and as the functional
${b}^{*}$ we take the perimeter functional $[U,V]\longrightarrow
o(U)-o(V)$. If $o([U,V]=0$ then it follows from the generalized
isoperimetric inequality, that $m([U,V])\leq 0$ and $m([U,V])=0$
only when $[U,V]=0$. Hence $m$ is negatively defined on the kernel
of the functional $o$.

\vspace{3mm}

Now we will present some more general version of the
Brun-Minkowski inequality. Let $[U,V]$ and $[P,Q]$ be two vectors
from $X_{\mathcal{S}}$ having a positive measure (i.e. such that
$m([U,V])>0$ and $m([P,Q])>0$). It follows from the considerations
made above for $\eta=m$ and $\widetilde{N}=\widetilde{M}$, that
\begin{equation}\label{BrunMink}
\widetilde{M}^{2}([U,V],[P,Q])\geq m([U,V])\cdot
m([P,Q]).
\end{equation}
or using only the measure $m$ in $\mathcal{S}$ and Minkowski addition we may write the inequality
(\ref{BrunMink}) in the following form:

$$(\frac{1}{2}\left(m(U+P)+m(V+Q)-m(U+Q)-m(V+P))\right)^{2}\geq$$
$$\geq (2m(U)+2m(V)-m(U+V))\cdot (2m(P)+2m(Q)-m(P+Q)).$$

\vspace{3mm}
  Inequality (\ref{BrunMink} may be considered as a
generalization of the classical Brunn-Minkowski inequality.
Indeed, for any $x\in X_{\mathcal{S}}\ni y$ we have: $2\cdot
\widetilde{M}(x,y)=m(x+y)-m(x)-m(y).$ Hence (\ref{BrunMink}) can
be now rewritten in the form:

\vspace{3mm}

 When $m(x)>0$ and $m(y)>0$
then

\begin{equation}\label{stare12.2}
\left(\frac{1}{2}(m(x+y)-m(x)-m(y)\right)^{2}\geq m(x)\cdot m(y)
\end{equation}
The classical Brunn-Minkowski inequality written in the form

$$\sqrt{m(x+y)}\geq \sqrt{m(x)}+ \sqrt{m(y)}$$

gives

$$m(x+y)\geq m(x)+ m(y) + 2\sqrt{m(x)\cdot m(y)}$$

end equivalently

$$(\frac{1}{2}(m(x+y)-m(x)-m(y))^2\geq m(x)\cdot m(y),$$
which is identical with (\ref{stare12.2}).

 It is clear, that (\ref{stare12.2}) is true also,
when at least one of vectors $x$ or $y$ has non-negative measure.
However (\ref{stare12.2}) may not be true, when both $x$ and $y$
have negative measure.

\vspace{3mm}

It is also easy to check, that the equality in the generalized
Brunn-Minkowski inequality, i.e. in the inequality
(\ref{BrunMink}) holds if and only if when $x=[U,V]$ and $y=[P,Q]$
are homothetic. Namely, let $\mathbb{L}$ be a two dimensional
subspace of $X_{\mathcal{S}}$. Then $\mathbb{L}$ has non-trivial
intersection with the kernel of the perimeter functional $o$. This
implies, that at least for one vector $w\in \mathbb{L}$ we have
$m(w)<0$. If we know that for some vector $u\in \mathbb{L}$ there
is $m(u)>0$ then $m$ restricted to $\mathbb{L}$ is a hyperbolic
form on $\mathbb{L}$. This means that there exists a linear
isomorphism
$$T:{\mathbb{R}}^2\ni (x_1,x_2)\longrightarrow T(x_1,x_2)\in
\mathbb{L}$$such that
$$m(T(x_1,x_2))=x_1^2-x_2^2.$$
The equality in the generalized Brunn-Minkowski inequality holds
if and only if the trinomial $f(t)=m(x+t\cdot y)$ has exactly one
root. But this is possible only when $x$ and $y$ are linearly
dependent.

\vspace{7mm}
\section{Part IV}

\vspace{3mm}
 {\bf  Connection to Hilbert space.}
\vspace{3mm}
\subsection{ Definition of an inner product}

We shall return now the space  $ X_{\mathcal{S}}$. We define in
this space a bilinear form given by the following formula:
\begin{equation}\label{iloczynskal}
 <[U,V];[P,Q]> = 2o([U,V])\cdot o([P,Q]) - 4\pi\cdot
\widetilde{M}([U,V]);[P,Q])
\end{equation}

The form (\ref{iloczynskal}) is in fact bilinear, since the
perimeter functional $o$ is linear and the bilinearity of
$\widetilde{M}$ was proved in Part I.

We observe now that this form is positively defined on
$X_\mathcal{S}$.
 Indeed if $$0 = <[U,V];[U,V]> = 2o^2([U,V])-
\widetilde{M}([U,V];[U,V]) = $$
$$= o^2([U,V]) + (o^2([U,V]) - \widetilde{M}([U,V];[U,V])) = $$
$$o^2([U,V]) + D([U,V]),$$ where $D$ is the deficit term defined
by Formula (\ref{deficit}). Hence $o^2([U,V])=0$ and $D([U,V])=0$.
It follows from  Theorem \ref{rownosc}, that $([U,V])= rB$ where
$B$ is the unit disc. Since $0=o^2([U,V])= r^2o^2(B)$ then $r=0$
and in consequence $([U,V])=0$.

Hence $(X_{\mathcal{S}}; <;>)$ is an unitary space, which after
completion (if necessary) gives a model of separable Hilbert
space. If one wants to have  the unique disk with the norm 1, some
renorming coefficient is needed. Namely, such a norm has a form:

\begin{equation}\label{renorm}
||[U,V]||^{2} = \frac{1}{4\pi^{2}}(2o^{2}([U,V])-4\pi\cdot
m([U,V])).
\end{equation}
The corresponding inner product has the form

\begin{equation}\label{renormbis}
<[U,V];[P,Q]>= \frac{1}{4\pi^{2}}(2(o([U,V])\cdot (o([P,Q])- 2\pi
(m(U+P)+m(V+Q)-m(U+Q)-m(V+P)).
\end{equation}
\vspace{3mm}

\subsection{The constructed space is an RKHS}

\vspace{5mm}

The elements of the space $X_{\mathcal{S}}$ are the equivalent
classes of the pairs $(U,V)$ of convex and centrally symmetric
sets $U$ and $V$. In appears, that the vectors from
$X_{\mathcal{S}}$ may be also considered as some  periodic and
continuous real functions on $\mathbb R$. More precisely, let
$\mathcal{F}$ denotes the space of all periodic (with the period
$\pi$) continuous real functions equipped with standard addition
and scalar multiplication.  Let $\varphi\in [0,\pi]$ and let
$I^{\varphi}$ be a diangle, whose argument is $\varphi$. We have
considered above the width functionals associated with a diangle
defined in subsection \ref{width}. Using this functionals we can
prove the following:

\vspace{5mm}

\begin{proposition}\label{stare14.1}
 The map

$$\omega:X_{\mathcal{S}}\ni [U,V]\longrightarrow
\overline{U}(I^{\varphi})-\overline{V}(I^{\varphi}):=[U,V](\varphi)\in
\mathcal{F},$$ is linear and injective.
\end{proposition}

\vspace{3mm}

\begin{proof}
The independence on the choice of a representative follows
directly from the linearity of the width functionals on
$\mathcal{S}$ and from the definition of the equivalence relation
$\diamond$. The injectivity is a consequence of the Radstrom lemma
(\cite{Rad}) and the linearity of $\omega$ follows directly from
the definition of addition and scalar multiplication. Also
continuity and periodicity are easy to check.

\end{proof}

 It will be more convenient to consider the vectors from
 $X_{\mathcal{S}}$ as the continuous functions $f$ on the interval
 $\Delta =[0,\pi]$ such that $f(0)=f(\pi)$. We shall prove that the space
 $X_{\mathcal{S}}$ is a  {\it reproducing kernel Hilbert space} (RKHS for short).
  The reproducing kernel Hilbert spaces was discovered at the
  beginning of XX-th century by S. Zaremba. A general theory of
  RKHS was formulated by N. Aronszajn in \cite{Aro}. An elegant
  introduction to this theory is in F.H. Szafraniec book
  \cite{Szaf} (in polish). The
 necessary definitions concerning RKHS  are to be found in
 \cite{Pau}.

 Let us fix $\varphi\in [0,\pi]$ and consider a function

\begin{equation}\label{stare14.3}
k_{\varphi}:[0,\pi]\ni \psi\longrightarrow
 k_{\varphi}(\psi)=[2B,\frac{\pi}{2}I^{\varphi}](\psi),
\end{equation}

 where $B$ is the unit disc and $I^{\varphi}$ is a diangle defined
 by the vector ${\bf{v}}=(cos\varphi, sin\varphi)$.

\vspace{3mm}

\begin{proposition}\label{kernelfunctions}
 For each $[U,V]\in X_{\mathcal{S}}$ the following equality holds:

$$\overline{U}(\varphi)-\overline{V}(\varphi) =
<[U,V];[2B,\frac{\pi}{2}I^{\varphi}]>.$$
\end{proposition}
\vspace{3mm}

Let us denote $d=\frac{\pi}{2}$ and $I^{\varphi}= I$. Using the
formula (\ref{renormbis}) we obtain:

$$<[U,V];[2B,\frac{\pi}{2}I^{\varphi}]>=  <[U,V];[2B,dI]> = $$
$$=\frac{1}{4\pi^{2}}(2(o([U,V])\cdot (o([2B,dI])- 2\pi(m(U+2B)+m(V+dI)-m(U+dI)-m(V+2B))=$$
$$=\frac{1}{4\pi^{2}}(2(o(U)-o(V))\cdot 2\pi - 2\pi (m(U)+4\pi +
2o(U) -m(V)-4\pi - 2o(V)+$$
$$+ m(V) + 2\cdot 2d \cdot \overline{V}(I) - m(U) - 2\cdot 2d \cdot
\overline{U}(I)))=$$
$$= \frac{1}{4\pi^{2}}((4\pi(o(U)-o(V))-
2\pi(2(o(U)-o(V))-2\pi(\overline{U}(I)-\overline{V}(I))) = $$
$$=\frac{1}{4\pi^{2}}((4\pi(o(U)-o(V))-(4\pi(o(U)-o(V)) +
4\pi^{2}(\overline{U}(I)-\overline{V}(I)))=$$
$$=\overline{U}(I)-\overline{V}(I).$$

In the language of RHKS the functions $k_{\varphi}$ are called
{\it the kernel functions} and  Proposition \ref{kernelfunctions}
asserts precisely, that the space $X_{\mathcal{S}}$ with the inner
product given by (\ref{renormbis}) has the {\it reproducing
property}. Now we shall calculate the {\it reproducing kernel} of
this space. As we know, the reproducing kernel is - in our case -
a function $K:\Delta\times \Delta \longrightarrow \mathbb R$ given
by the formula:
$$K(\varphi,\psi)=<k_{\varphi},k_{\psi}>.$$

Let us calculate.

$$ <k_{\varphi},k_{\psi}> =
<[2B,\frac{\pi}{2}I^{\varphi}];[2B,\frac{\pi}{2}I^{\psi}]>=$$
$$= \frac{1}{4\pi^{2}}(2\cdot 2\pi\cdot 2\pi -
4\pi\cdot\frac{1}{2}\cdot(m(2B+2B)+m(\frac{\pi}{2}I^{\varphi}+\frac{\pi}{2}I^{\psi})
-m(2B+\frac{\pi}{2}I^{\varphi})- m(2B+\frac{\pi}{2}I^{\psi}))=$$
$$=\frac{1}{4\pi^{2}}(8\pi^{2}-2\pi(16\pi + \frac{\pi^{2}}{4}\cdot
4\cdot sin|\varphi-\psi|-4\pi-4\pi-4\pi-4\pi))=$$
$$=\frac{1}{4\pi^{2}}(8\pi^{2}-2\pi^{3}\cdot
sin|\varphi-\psi|))=$$
$$=2-\frac{\pi}{2}\cdot sin|\varphi-\psi|.$$

Hence we have proved the following:

\vspace{3mm}
\begin{theorem}\label{stare14.5}

 The space $(X_{\mathcal{S}};<;>)$ is a reproducing kernel Hilbert
space on $[0,\pi]$ and its kernel is given by the formula
$$K(\varphi,\psi)=2-\frac{\pi}{2}\cdot sin|\varphi-\psi|.$$
\end{theorem}
\vspace{3mm}

It follows from the reproducing property, that each {\it
evaluation functional} is bounded. In our case this means that for
each $\varphi\in [0,\pi]$ there exists a constant (in general
depending on $\varphi$) $C(\varphi)$, such that for each $[U,V]\in
X_{\mathcal{S}}$ there is:
$$E_{\varphi}([U,V])=
\overline{U}(\varphi)-\overline{V}(\varphi)\leq C(\varphi)\cdot
||[U,V]||.$$

Since, as we have observed in Proposition \ref{kernelfunctions},
the evaluation functional are given by the formula

 $$E_{\varphi}([U,V])=\overline{U}(\varphi)-\overline{V}(\varphi) =
<[U,V];[2B,\frac{\pi}{2}I^{\varphi}]>,$$ then, using the Schwarz
inequality, we have
\begin{equation}\label{stare14.6}
|E_{\varphi}([U,V])|\leq ||[U,V]||\cdot
||[2B,\frac{\pi}{2}I^{\varphi}]||\leq \sqrt{2}||[U,V]||.
\end{equation}
This follows from the equality

$$||[2B,\frac{\pi}{2}I^{\varphi}]||^{2}=2,$$ which is easy to
check. We see that the constant $C(\varphi)$ does not depend on
$\varphi$, which is clear, since the norm in $X_{\mathcal{S}}$ is
invariant under rotations in the plane. This fact has an important
consequence. It is known, that the space $X_{\mathcal{S}}$ can be
equipped with another norm given by the formula:

$$||[U,V]||_{c}=\sup\left\{|\overline{U}(\varphi)-\overline{V}(\varphi)|:\varphi\in
[0,\pi]\right\}.$$

This is a norm in $X_{\mathcal{S}}$ and as it can be easily proved
\begin{equation}\label{normaciagla}
||[U,V]||_{c}=\rho_{H}(U,V)
\end{equation}
where $\rho_{H}$ is a Hausdorff distance in the space of compact
sets. It follows from  inequality (\ref{stare14.6}),  that for
each $[U,V]\in X_{\mathcal{S}}$ we have

\begin{equation}\label{oszacowanienormy}
||[U,V]||_{c}\leq \sqrt{2}||[U,V]||.
\end{equation}

It can be proved, that the space $X_{\mathcal{S}}$ equipped with
the norm (\ref{normaciagla}) after completion is isomorphic to a
space $C(K)$ of continuous function on some compact set $K$. One
considers also another norm on $X_{\mathcal{S}}$, called {\it
Bartels-Pallaschke norm} (\cite{GPU}), defined as follows 
\begin{equation}\label{normasilna}
||[U,V]]||_{s}=\inf\left\{(||P||_c + ||Q||_c):(U,V)\diamond (P,Q)\right\}
\end{equation}

This norm (\ref{normasilna}) is stronger and not equivalent to the
hilbertian norm (\ref{renorm}). One can also show, that the space
$(X_{\mathcal{S}}, ||[.]||_s)$ is complete. Hence the norm
(\ref{renorm}) is not complete.

It remains to solve the so-called {\it reconstruction problem,}
(\cite{Pau}), i.e. to describe those continuous functions on
$[0,\frac{\pi}{2}]$, which belongs to the hilbertian completion of
$X_{\mathcal{S}}$. I was not able to do it, and maybe this will
appear more difficult.

\end{document}